\documentclass{amsart}

\newtheorem{theorem}{Theorem}[section]
\newtheorem{lemma}[theorem]{Lemma}
\newtheorem{proposition}[theorem]{Proposition}

\theoremstyle{definition}

\newtheorem{example}[theorem]{Example}

\newtheorem{remark}[theorem]{Remark}

\usepackage{amscd,amssymb}
\begin{document}

\title[Classical invariant theory for free metabelian Lie algebras]
{Classical invariant theory\\
for free metabelian Lie algebras}

\author[Vesselin Drensky, {\c S}ehmus F{\i}nd{\i}k]
{Vesselin Drensky, {\c S}ehmus F{\i}nd{\i}k}
\address{Institute of Mathematics and Informatics,
Bulgarian Academy of Sciences,
1113 Sofia, Bulgaria}
\email{drensky@math.bas.bg}
\address{Department of Mathematics,
\c{C}ukurova University, 01330 Balcal\i,
 Adana, Turkey}
\email{sfindik@cu.edu.tr}

\thanks
{Partially supported by Grant I02/18 of the Bulgarian Science Fund}
\thanks
{The research of the first named author was partially supported
by Grant Ukraine 01/0007 of the Bulgarian Science Fund for Bilateral Scientific Cooperation between Bulgaria and Ukraine}
\thanks
{The research of the second named author was partially supported by the
Council of Higher Education (Y\"OK) in Turkey}


\begin{abstract}
Let $W_d=K^d$ be the $d$-dimensional vector space over a field $K$ of characteristic 0
with the canonical action of the general linear group $GL_d(K)$
and let $KX_d$ be the vector space of the linear functions on $W_d$.
One of the main topics of classical invariant theory is the study
of the algebra of invariants $K[X_d]^{SL_2(K)}$ of the special linear group $SL_2(K)$,
when $KX_d$ is a direct sum of $SL_2(K)$-modules of binary forms.
Noncommutative invariant theory deals with the algebra of invariants $F_d({\mathfrak V})^G$
of a group $G<GL_d(K)$ acting on the relatively free algebra $F_d({\mathfrak V})$
of a variety of $K$-algebras $\mathfrak V$.
Due to the noncommutativity it is more convenient to assume that $F_d({\mathfrak V})$ is generated by $W_d$
instead of by $KX_d$, with the corresponding action of $GL_d(K)$.
In this paper we consider the free metabelian Lie algebra $F_d({\mathfrak A}^2)$
which is the relatively free algebra in the variety ${\mathfrak A}^2$ of metabelian (solvable of class 2) Lie algebras.
We study the algebra $F_d({\mathfrak A}^2)^{SL_2(K)}$ and
describe the cases when it is finitely generated.
This happens if and only if as an $SL_2(K)$-module $W_d\cong K^2\oplus K\oplus\cdots\oplus K$ or $W_d\cong S^2(K^2)$
(and in the trivial case $KX_d\cong K\oplus\cdots\oplus K$).
Here $SL_2(K)$ acts canonically on $K^2$, trivially on $K$, and $S^2(K^2)$ is the symmetric square of $K^2$.
For small $d$ we give a list of generators even when $F_d({\mathfrak A}^2)^{SL_2(K)}$ is not finitely generated.
The methods for establishing that the algebra $F_d({\mathfrak A}^2)^{SL_2(K)}$ is not finitely generated work also for other
relatively free algebras $F_d({\mathfrak V})$ and for other groups $G$.
\end{abstract}

\maketitle

\section{Introduction}

In classical invariant theory one considers the action of the general linear group $GL_d({\mathbb C})$ on the algebra ${\mathbb C}[X_d]$
of the polynomial functions on the $d$-dimensional vector space $W_d={\mathbb C}^d$.
Then for a subgroup $G$ of $GL_d({\mathbb C})$ one studies the algebra ${\mathbb C}[X_d]^G$ of $G$-invariants.
One of the possible noncommutative generalizations is to replace the action of $GL_d({\mathbb C})$ on ${\mathbb C}[X_d]$
with the action of $GL_d(K)$ on the relatively free algebra $F_d({\mathfrak V})$
in a variety of (associative, Lie, Jordan, etc.) $K$-algebras $\mathfrak V$, where $K$ is an arbitrary field of characteristic 0.
Due to the noncommutativity it is more convenient to assume that $F_d({\mathfrak V})$ is generated by $W_d$
instead of by $KX_d$ as in classical invariant theory, with the corresponding action of $GL_d(K)$.
In the sequel we consider this action of $GL_d(K)$.
In particular, the role of $K[X_d]$ is played by the symmetric algebra
\[
S(W_d)=S(KY_d)=\bigoplus_{n\geq 0}S^n(W_d),
\]
where $S^n(W_d)$ is $n$th symmetric power of the canonical $GL_d(K)$-module $W_d=KY_d$ with the diagonal action of $GL_d(K)$ on $S^n(W_d)$.
Then for a subgroup $G$ of $GL_d(K)$
we study the algebra $F_d({\mathfrak V})^G$ of $G$-invariants in $F_d({\mathfrak V})$.

One of the most intensively studied objects in classical invariant
theory is the algebra of invariants of the special linear group
$SL_2({\mathbb C})$ acting on binary forms.
The translation in the language of noncommutative invariant theory is to assume that as an $SL_2(K)$-module
the vector space $W_d=KY_d$, $d\geq 2$, is isomorphic to a direct sum $V_{p_1}\oplus\cdots\oplus V_{p_r}$,
where $V_p=S^p(K^2)$ and $SL_2(K)$ acts canonically on $K^2$
(and the action of $SL_2(K)$ on $KY_d$ is not trivial, i.e., at least one $p_i$ is positive).
Then we extend diagonally the action of $SL_2(K)$ on
the relatively free algebra $F_d({\mathfrak V})$ generated by $Y_d$ in the variety of $K$-algebras $\mathfrak V$.
The algebra $F_d({\mathfrak V})^{SL_2(K)}$ is a noncommutative analogue of the ``classical'' object
$K[X_d]^{SL_2(K)}$ of polynomial $SL_2(K)$-invariants.

In the general case when $G$ is an arbitrary subgroup of $GL_d(K)$ and $\mathfrak V$ is an arbitrary variety
of algebras not too much is known about the algebra of $G$-invariants $F_d({\mathfrak V})^G$ in $F_d({\mathfrak V})$.
The picture is more or less clear for varieties of associative algebras. For finite groups $G$ see the surveys
\cite{F, D2, KS} and for $G$ reductive -- the papers \cite{V, DD}. In particular, $F_d({\mathfrak V})^G$
is finitely generated for all reductive groups $G$ acting rationally on $KX_d$ if and only if $\mathfrak V$
satisfies the Engel identity $[[[z_2,z_1],\ldots],z_1]=0$.

It is well known that for varieties of Lie algebras there is a dichotomy:
The variety $\mathfrak V$ either contains the metabelian variety ${\mathfrak A}^2$ or satisfies the Engel identity. Recall that
the variety ${\mathfrak A}^2$ of metabelian (solvable of class 2) Lie algebras is defined by the polynomial identity
\[
[[z_1,z_2],[z_3,z_4]]=0
\]
and the free metabelian Lie algebra $F_d({\mathfrak A}^2)$
is isomorphic to the factor algebra $L_d/L_d''$ of the free $d$-generated Lie algebra $L_d$
modulo the second term $L_d''=[[L_d,L_d],[L_d,L_d]]$ of its derived series.
If $\mathfrak V$ satisfies the Engel identity, then the celebrated result of Zel'manov \cite{Z} gives that $\mathfrak V$
is nilpotent. (The fact that the Engel identity implies that $F_d({\mathfrak V})$ is nilpotent
with the class of nilpotency depending on $d$ follows from the
famous paper by Kostrikin \cite{Ko} on the Burnside problem.)
When $\mathfrak V$ is nilpotent, the algebra $F_d({\mathfrak V})$ is finite dimensional. Hence $F_d({\mathfrak V})^G$ is also finite
dimensional and therefore finitely generated as a Lie algebra for all groups $G$. If $\mathfrak V$ contains ${\mathfrak A}^2$, then
the canonical homomorphism $F_d({\mathfrak V})\to F_d({\mathfrak A}^2)$ maps $F_d({\mathfrak V})^G$ onto $F_d({\mathfrak A}^2)^G$.
Hence, concerning the finite generation of $F_d({\mathfrak V})^G$, the free metabelian algebra is the key object.
If the group $G$ is finite and $\mathfrak V$ contains ${\mathfrak A}^2$, then $F_d({\mathfrak V})^G$ is never finitely generated
by \cite{Br, D1}. To the best of our knowledge, not too much information is available for $F_d({\mathfrak V})^G$ when the group $G$ is not finite.
The algebra of invariants $F_d({\mathfrak V})^G$ of the group $G=UT_2(K)$ of unipotent $2\times 2$ matrices acting unipotently on $KX_d$
was studied in \cite{DG, DDF}. In particular, when ${\mathfrak V}={\mathfrak A}^2$ the algebra of $F_d({\mathfrak A}^2)^{UT_2(K)}$
is finitely generated if and only if $KX_d=K^2\oplus K^{d-2}$ and $UT_2(K)$ acts canonically on $K^2$ and trivially on $K^{d-2}$.

In the present paper we study the algebra $F_d({\mathfrak A}^2)^{SL_2(K)}$ of $SL_2(K)$-invariants of $F_d({\mathfrak A}^2)$.
We show that $F_d({\mathfrak A}^2)^{SL_2(K)}$ is finitely generated
if and only if as an $SL_2(K)$-module $KY_d\cong V_1\oplus V_0\oplus\cdots\oplus V_0$ or $KY_d\cong V_2$
(and in the obvious case $KY_d\cong V_0\oplus\cdots\oplus V_0$) with the trivial action of $SL_2(K)$ on $V_0\cong K$.

One of the main tools in our considerations is the Shmel'kin embedding theorem \cite{Sh} which
allows to consider $F_d({\mathfrak A}^2)$ as a subalgebra of the abelian wreath
product $(KA_d)\text{\rm wr}(KU_d)$ of two $d$-dimensional abelian Lie algebras $KA_d$ and $KU_d$ with bases
$A_d=\{a_1,\ldots,a_d\}$ and $U_d=\{u_1,\ldots,u_d\}$, respectively.

To see whether $F_d({\mathfrak A}^2)^{SL_2(K)}$ is finitely generated, the most
difficult part of the proof is to find whether the commutator ideal $F_d({\mathfrak A}^2)'$ contains nonzero invariants.
For small $d$ it can be established calculating the Hilbert (or Poincar\'e) series of $F_d({\mathfrak A}^2)^{SL_2(K)}$.
But it has turned out that in the general case we have to embed $KA_d\text{\rm wr}KU_d$ into a Poisson algebra and then to use classical results
for the existence of sufficiently many invariants of $SL_2(K)$.

The commutator ideal $F_d({\mathfrak A}^2)'$ has a natural structure of an $S(W_d)$-module. Well known results give that
$(F_d({\mathfrak A}^2)')^{SL_2(K)}$ is a finitely generated $S(W_d)^{SL_2(K)}$-module.
For small $d$ we give a list of generators of the $S(W_d)^{SL_2(K)}$-module $(F_d({\mathfrak A}^2)')^{SL_2(K)}$.
This allows to find explicit sets of generators of the Lie algebra
$F_d({\mathfrak A}^2)^{SL_2(K)}$ even if is not finitely generated. In this part our paper may be considered as a continuation of \cite{DDF}
where similar problems were solved for $F_d({\mathfrak A}^2)^{UT_2(K)}$.

Finally, we want to mention that the methods for establishing that the algebra $F_d({\mathfrak A}^2)^{SL_2(K)}$
is not finitely generated work also for other relatively free algebras $F_d({\mathfrak V})$ and for other groups $G$.

\section{Preliminaries}

In the sequel we fix the base field $K$ of characteristic 0.
All vector spaces, algebras, tensor products, etc., will be over $K$.
If $R$ is any algebra generated by the elements $r_1,\ldots,r_m$ and $I$ is an ideal of $R$,
we shall denote the generators of the factor algebra $R/I$ with the same symbols $r_1,\ldots,r_m$.
\subsection{Free metabelian Lie algebras and abelian wreath products.}
For a background on varieties of Lie algebras we refer to \cite{Ba}.
We denote by $F_d$ the free metabelian Lie algebra $F_d({\mathfrak A}^2)$ freely generated by $Y_d=\{y_1,\ldots,y_d\}$.
We assume that all Lie commutators are left normed, e.g.,
\[
[y_1,y_2,y_3]=[[y_1,y_2],y_3]=[y_1,y_2]\text{ad}y_3.
\]
The commutator ideal $F_d'$ has a basis consisting of all
\[
[y_{j_1},y_{j_2},y_{j_3},\ldots,y_{j_n}],\quad 1\leq j_i\leq d,\quad j_1>j_2\leq j_3\leq\cdots\leq j_n.
\]
The algebra $F_d$ has a grading
\[
F_d=(F_d)_1\oplus (F_d)_2\oplus (F_d)_3\oplus\cdots
\]
where $(F_d)_n$ is the vector subspace of all homogeneous elements of degree $n$. Similarly, $F_d$ is
${\mathbb Z}^d$-graded (or multigraded), with grading which
counts the degree of each variable $y_i$, $i=1,\ldots,d$.
The Hilbert series of $F_d$ is the formal power series defined by
\[
H(F_d,T_d)=H(F_d,t_1,\ldots,t_d)=\sum_{n_i\geq 0}\dim (F_d)_nT_d^n,
\]
where $(F_d)_n$ is the homogeneous component of degree $n=(n_1,\ldots,n_d)$ and
$T_d^n=t_1^{n_1}\cdots t_d^{n_d}$. It is well known, see e.g., \cite{D1}, that
\[
H(F_d,T_d)=1+(t_1+\cdots+t_d)+(t_1+\cdots+t_d-1)\prod_{j=1}^d\frac{1}{1-t_j}.
\]
The consequence of the metabelian identity
\[
[z_{j_1},z_{j_2},z_{j_{\sigma(3)}},\ldots,z_{j_{\sigma(n)}}]
=[z_{j_1},z_{j_2},z_{j_3},\ldots,z_{j_n}],
\]
where $\sigma$ is an arbitrary permutation of $3,\ldots,n$,
allows to define an action of the symmetric algebra $S(W_d)=S(KY_d)$ on $F_d'$ by the rule
\[
wf(y_1,\ldots,y_d)=wf(\text{ad}y_1,\ldots,\text{ad}y_d),\quad w\in F_d',\quad f(Y_d)\in S(KY_d).
\]
One of the main tools in our considerations is the Shmel'kin embedding theorem \cite{Sh} which
allows to consider $F_d({\mathfrak A}^2)$ as a subalgebra of the abelian wreath
product $KA_d\text{\rm wr}KU_d$ of two $d$-dimensional abelian Lie algebras $KA_d$ and $KU_d$.
Recall the construction of abelian wreath products due to Shmel'kin \cite{Sh}.
Let $KA_d$ and $KU_d$ be the abelian
Lie algebras with bases $A_d=\{a_1,\ldots,a_d\}$ and
$U_d=\{u_1,\ldots,u_d\}$, respectively. Let $C_d$ be the free right
$S(KU_d)$-module with free generators $a_1,\ldots,a_d$.
We give it the structure of a Lie algebra with trivial multiplication.
Then $(KA_d)\text{\rm wr}(KU_d)$ is equal to the semidirect sum $C_d\leftthreetimes (KU_d)$. The elements
of $(KA_d)\text{\rm wr}(KU_d)$ are of the form
\[\sum_{j=1}^da_jf_j(U_d)+\sum_{j=1}^d\beta_ju_j,\quad
f_j\in S(KU_d),\quad \beta_j\in K,\quad j=1,\ldots,d.
\]
The multiplication in $(KA_d)\text{\rm wr}(KU_d)$ is defined by
\[
[C_d,C_d]=[U_d,U_d]=0,
\]
\[
[a_jf_j(U_d),u_i]=a_jf_j(U_d)u_i,\quad i,j=1,\ldots,d.
\]
The embedding theorem of Shmel'kin \cite{Sh} gives that the free metabelian Lie algebra $F_d$
is isomorphic to the subalgebra of $(KA_d)\text{\rm wr}(KU_d)$ generated by
$a_j+u_j$, $j=1,\ldots,d$. In the sequel we identify
$F_d$ with its copy in $(KA_d)\text{\rm wr}(KU_d)$ and assume that
\[
y_j=a_j+u_j,\quad j=1,\ldots,d.
\]
If $w$ belongs to $F_d'$,
\[
w=\sum_{i>j}[y_i,y_j]f_{ij}(Y_d)=\sum_{i>j}[y_i,y_j]f_{ij}(\text{ad}y_1,\ldots,\text{ad}y_d),\quad f_{ij}(Y_d)\in S(KY_d),
\]
then in $(KA_d)\text{\rm wr}(KU_d)$
\[
w=\sum(a_iu_j-a_ju_i)f_{ij}(U_d).
\]
The element
\[
\sum_{i=1}^da_if_i(U_d)\in (KA_d)\text{\rm wr}(KU_d)
\]
belongs to $F_d'$ if and only if
\[
\sum_{i=1}^du_if_i(U_d)=0.
\]

\subsection{Poisson algebras.}\label{subsection on Poisson algebras}
Given two algebraically independent $SL_2(K)$-invariants $f_1(Y_d),f_2(Y_d)\in S(KY_d)$ we shall construct a nonzero
$SL_2(K)$-invariant $\pi(f_1,f_2)$ in $F_d'$. For this purpose we need to embed $(KA_d)\text{\rm wr}(KU_d)$ into a Poisson algebra.
Recall that the commutative-associative algebra $P$ with multiplication $\cdot$ is a (commutative) Poisson algebra
if it has a Poisson bracket $[\cdot,\cdot]$ such that
$(P,[\cdot,\cdot])$ is a Lie algebra and $[\cdot,\cdot]$ satisfies the Leibniz rule
\[
[f_1\cdot f_2,f_3]=[f_1,f_3]\cdot f_2+f_1\cdot[f_2,f_3],\quad f_1,f_2,f_3\in P.
\]
In the sequel we shall omit the $\cdot$ in the multiplication in Poisson algebras.

\begin{lemma}\label{Poisson structure on K[A,Y]}
Let $R_d$ be the factor algebra $S(KA_d\oplus KU_d)/(A_d^2)$ of the symmetric algebra $S(KA_d\oplus KU_d)$
on the vector space $KA_d\oplus KU_d$ with basis $A_d\cup U_d$ modulo the ideal generated
by all products $a_ia_j$, $i,j=1,\ldots,d$, endowed with the Poisson bracket defined by
\[
[a_iu_1^{p_1}\cdots u_d^{p_d},u_1^{q_1}\cdots u_d^{q_d}]=(q_1+\cdots+q_d)a_iu_1^{p_1+q_1}\cdots u_d^{p_d+q_d},
\]
\[
[u_1^{p_1}\cdots u_d^{p_d},u_1^{q_1}\cdots u_d^{q_d}]=[a_iu_1^{p_1}\cdots u_d^{p_d},a_ju_1^{q_1}\cdots u_d^{q_d}]=0.
\]
Then

{\rm (i)} As a vector space $R_d$ has a basis
\[
\{U_d^p=u_1^{q_1}\cdots u_d^{q_d},\quad a_iU_d^q=a_iu_1^{q_1}\cdots u_d^{q_d}\mid q_j\geq 0,\quad i=1,\ldots,d\}.
\]

{\rm (ii)} The Lie algebra $(KA_d)\text{\rm wr}(KU_d)$ is isomorphic to the Lie subalgebra of $R_d$ generated by $A_d$ and $U_d$.
\end{lemma}

\begin{proof}
The verification that $R_d$ is a Poisson algebra with respect to the given bracket is straightforward.
The defining relations $a_ia_j=0$ give that as a vector space $R_d$ has a basis consisting of the elements from the statement (i)
of the lemma, since it is a commutative associative algebra.
For the proof of (ii) we see that the commutator ideal of the Lie subalgebra of $R_d$
generated by $A_d$ and $U_d$ has a basis consisting of all
\[
a_iU_d^p=a_i(\text{ad}^{p_1}u_1)\cdots (\text{ad}^{p_d}u_d), \quad p=(p_1,\ldots,p_d)\not=(0,\ldots,0),
\]
the same as the commutator ideal of $(KA_d)\text{\rm wr}(KU_d)$. The multiplication table is also the same. Hence both Lie algebras are isomorphic.
\end{proof}

\subsection%{The Jacobian criterion for algebraic independence.}
{Algebraically independent elements in Poisson algebras.}
The well known Jacobian criterion \cite{J} from 1841 gives that the elements $f_1,\ldots,f_m\in S(KY_d)$, $m\leq d$,
are algebraically dependent if and only if
\[
\left\vert\begin{matrix}
\partial f_1/\partial y_{j_1}&\partial f_2/\partial y_{j_1}&\cdots&\partial f_m/\partial y_{j_1}\\
\partial f_1/\partial y_{j_2}&\partial f_2/\partial y_{j_2}&\cdots&\partial f_m/\partial y_{j_2}\\
\vdots&\vdots&\ddots&\vdots\\
\partial f_1/\partial y_{j_m}&\partial f_2/\partial y_{j_m}&\cdots&\partial f_m/\partial y_{j_m}\\
\end{matrix}\right\vert=0
\]
for all determinants
with $1\leq j_1<j_2<\cdots<j_m\leq d$. A simple proof for $m=d$ can be found e.g., in \cite{ER} or \cite{SU}.
We shall use the Jacobian criterion for $m=2$ only (and for any $d\geq 2$):

\begin{lemma}\label{Jacobian criterion}
Let $f_1(Y_d),f_2(Y_d)\in S(KY_d)$ and let
\[
J_{ij}(f_1,f_2)=\left\vert\begin{matrix}
\partial f_1/\partial y_i&\partial f_2/\partial y_i\\
\partial f_1/\partial y_j&\partial f_2/\partial y_j\\
\end{matrix}\right\vert=0
\]
for all $1\leq i<j\leq d$. Then $f_1(Y_d)$ and $f_2(Y_d)$ are algebraically dependent.
\end{lemma}

The free Poisson algebra $P(Y_d)$ freely generated by $Y_d$ can be obtained in the following way.
Fix a basis $B_d=\{b_1,b_2,\ldots\}$ of the free Lie algebra $L_d$. Then, as a commutative algebra, $P(Y_d)$ is isomorphic to
the symmetric algebra $S(L_d)=S(KB_d)$. The Poisson bracket in $P(Y_d)$ is defined
extending the Lie bracket in $L_d$ using the Leibniz rule.
If $f_1(Y_d),f_2(Y_d)$ are elements in $S(KY_d)$ considered as elements of the free Poisson algebra $P(Y_d)$, then, see e.g., \cite{SU},
\begin{equation} \label{eq:1}
[f_1,f_2]=\sum_{i<j}[y_i,y_j]\left(\frac{\partial f_1}{\partial y_i}\frac{\partial f_2}{\partial y_j}
-\frac{\partial f_1}{\partial y_j}\frac{\partial f_2}{\partial y_i}\right).
\end{equation}
Hence $f_1(Y_d),f_2(Y_d)\in S(KY_d)\subset P(Y_d)$ are algebraically independent in $S(KY_d)$ if and only if $[f_1,f_2]\not=0$
in $P(Y_d)$. An analogue for any $f_1,f_2\in P(Y_d)$ was established in \cite{MLU}.
We cannot apply directly this result for the algebra $R_d$ defined in the previous subsection because it is a homomorphic image
of the free Poisson algebra. We believe that the following proposition is of independent interest
because it allows, starting from two polynomials in commuting variables, to construct elements with prescribed properties in the free
metabelian Lie algebra.

\begin{proposition}\label{independent polynomials in R_d}
Let the homogeneous polynomials $f_1(Y_d)$ and $f_2(Y_d)$ be algebraically independent in $S(KY_d)$. Then in the Poisson algebra $R_d$
the commutator
\[
\pi(f_1,f_2)=[f_1(A_d+U_d),f_2(A_d+U_d)]
\]
\[
=[f_1(a_1+u_1,\ldots,a_d+u_d),f_2(a_1+u_1,\ldots,a_d+u_d)]
\]
is different from zero and belongs to the commutator ideal $F_d'$ of the free metabelian Lie algebra $F_d$ embedded in $R_d$.
\end{proposition}

\begin{proof}
A direct consequence of the equality (\ref{eq:1}) together with the fact that we may view the Poisson subalgebra of $R_d$
generated by $a_1+u_1,\ldots,a_d+u_d$ as a Poisson homomorphic image of the free Posisson algebra,
we obtain
\[
[f_1(A_d+U_d),f_2(A_d+U_d)]=\sum_{i<j}(a_iu_j-a_ju_i)\left\vert\begin{matrix}
\partial f_1(U_d)/\partial u_i&\partial f_2(U_d)/\partial u_i\\
\partial f_1(U_d)/\partial u_j&\partial f_2(U_d)/\partial u_j\\
\end{matrix}\right\vert
\]
\[
=\sum_{i<j}(a_iu_j-a_ju_i)\left(\frac{\partial f_1(U_d)}{\partial u_i}\frac{\partial f_2(U_d)}{\partial u_j}
-\frac{\partial f_1(U_d)}{\partial u_j}\frac{\partial f_2(U_d)}{\partial u_i}\right)
\]
\[
=\sum_{i=1}^da_i\left(\frac{\partial f_1(U_d)}{\partial u_i}\sum_{j=1}^du_j\frac{\partial f_2(U_d)}{\partial u_j}
-\sum_{j=1}^d\frac{\partial f_1(U_d)}{\partial u_j}u_j\frac{\partial f_2(U_d)}{\partial u_i}\right).
\]
because $[a_i+u_i,a_j+u_j]=a_iu_j-a_ju_i$ and $a_ka_l=0$.
Since $f_2(U_d)$ is homogeneous, it is easy to see that
\[
\sum_{j=1}^du_j\frac{\partial f_2(U_d)}{\partial u_j}=\deg(f_2)f_2(U_d).
\]
If $[f_1(A_d+U_d),f_2(A_d+U_d)]=0$, then the coordinates of $a_i$ are equal to 0 and we obtain a homogeneous linear system
\[
\frac{\partial f_1(U_d)}{\partial u_i}\deg(f_2)f_2(U_d)-\sum_{j=1}^d\frac{\partial f_1(U_d)}{\partial u_j}y_j\frac{\partial f_2(U_d)}{\partial u_i}=0
\]
of $d$ equations, for $i=1,\ldots,d$, with unknowns $\partial f_1(U_d)/\partial u_j$, $j=1,\ldots,d$. The matrix of this system is
\[
\left(\begin{matrix}
\deg(f_2)f_2-u_1\partial f_2/\partial u_1&-u_2\partial f_2/\partial u_1&\cdots&-u_d\partial f_2/\partial u_1\\
-u_1\partial f_2/\partial u_2&\deg(f_2)f_2-u_2\partial f_2/\partial u_2&\cdots&-u_d\partial f_2/\partial u_2\\
\vdots&\vdots&\ddots&\vdots\\
-u_1\partial f_2/\partial u_d&-u_2\partial f_2/\partial u_d&\cdots&\deg(f_2)f_2-u_d\partial f_2/\partial u_d\\
\end{matrix}\right).
\]
Since $f_2(U_d)\not=0$, at least one partial derivative if nonzero. Let, for example, $\partial f_2/\partial u_d\not=0$.
Subtracting the last row of the system multiplied by $(\partial f_2/\partial u_i)/(\partial f_2/\partial u_d)$ from the $i$-th row,
we obtain an equivalent system with a matrix
\[
\left(\begin{matrix}
\deg(f_2)f_2&0&\cdots&0&\ast\\
0&\deg(f_2)f_2&\cdots&0&\ast\\
\vdots&\vdots&\ddots&\vdots&\vdots\\
0&0&\cdots&\deg(f_2)f_2&\ast\\
-u_1\partial f_2/\partial u_d&-u_2\partial f_2/\partial u_d&\cdots&-u_{d-1}\partial f_2/\partial u_d&\deg(f_2)f_2-u_d\partial f_2/\partial u_d\\
\end{matrix}\right).
\]
Since $f_2(U_d)\not=0$, the rank of the matrix is $\geq d-1$.
Since the system has a nonzero solution $\partial f_2/\partial u_j=\partial f_1/\partial u_j$, $j=1,\ldots,d$, the rank is
$d-1$. Hence all solutions of the system are
\[
\frac{\partial f_2}{\partial u_j}=\frac{p(U_d)}{q(U_d)}\frac{\partial f_1}{\partial u_j},\quad j=1,\ldots,d,
\]
for some polynomials $p(U_d),q(U_d)\in S(KU_d)$. Therefore,
\[
J_{ij}(f_1(U_d),f_2(U_d))=0,\quad i,j=1,\ldots,d,
\]
which contradicts with the algebraic independence of $f_1(U_d)$ and $f_2(U_d)$.
Clearly, the form of the element
\[
[f_1(A_d+U_d),f_2(A_d+U_d)]=\sum_{i<j}(a_iu_j-a_ju_i)f_{ij}(U_d),\quad f_{ij}(U_d)\in S(KU_d),
\]
shows that $[f_1(A_d+U_d),f_2(A_d+U_d)]$ is a Lie element and belongs to $F_d'$.
\end{proof}

\subsection{Invariant theory of the special linear group.}
All necessary information on representation theory of the special linear group $SL_2(K)$ and the general linear group
$GL_2(K)$ and on invariant theory of $SL_2(K)$ can be found in many books, see e.g., \cite{Sp, W}. For
a background on symmetric functions see, e.g., \cite{Md}.
In our considerations we combine ideas of De Concini, Eisenbud, Procesi \cite{DEP} and for the noncommutative generalizations of
Almkvist, Dicks, Formanek \cite{ADF}, see \cite{BBD} for details.
Every rational representation
$\rho:SL_2(K)\to GL_d(K)$ of $SL_2(K)$ is a direct sum of irreducible representations. For each nonnegative integer $p$
there exists a unique irreducible rational $(p+1)$-dimensional $SL_2(K)$-representation $\rho_p$.
The corresponding $SL_2(K)$-module is $V_p=S^p(K^2)$. If the vector space $K^2$ has a basis $\{e_1,e_2\}$, then
$V_p$ has a basis $\{e_1^p,e_1^{p-1}e_2,\ldots,e_1e_2^{p-1},e_2^p\}$. We equip $V_p$ also with a structure of a $GL_2(K)$-module
and use the same notation $\rho_p$ for the corresponding $GL_2(K)$-representation.
If
\[
g=\left(\begin{matrix}
\gamma_{11}&\gamma_{12}\\
\gamma_{21}&\gamma_{22}\\
\end{matrix}\right)\in GL_2(K).
\]
then
\[
g:\sum_{j=0}^p\alpha_je_1^{p-j}e_2^j\to
\sum_{j=0}^p\alpha_j(\gamma_{11}e_1+\gamma_{21}e_2)^{p-j}(\gamma_{12}e_1+\gamma_{22}e_2)^j,\quad \alpha_1,\alpha_2\in K.
\]

\begin{example}\label{action of g1 and g2}
The matrices
\[
g_1=\left(\begin{matrix}
1&1\\
0&1\\
\end{matrix}\right),
g_2=\left(\begin{matrix}
1&0\\
1&1\\
\end{matrix}\right)
\]
in $SL_2(K)$ act on $V_p$ by
\[
g_1(e_1^{p-j}e_2^j)=\sum_{l=0}^j\binom{j}{l}e_1^{p-l}e_2^l,\quad j=0,1,\ldots,p,
\]
\[
g_2(e_1^je_2^{p-j})=\sum_{l=0}^j\binom{j}{l}e_1^le_2^{p-l},\quad l=0,1,\ldots,p.
\]
\end{example}

If $KY_d$ is isomorphic to a direct sum $V_{p_1}\oplus\cdots\oplus V_{p_r}$, we may
change linearly the variables and assume that the bases of the $SL_2(K)$-submodules $V_{p_i}$ of $KY_d$
are subsets of $Y_d$. If $KY_{p+1}=V_p$ we shall identify $y_1,y_2,\ldots,y_{p+1}$ with $e_1^p,e_1^{p-1}e_2,\ldots,e_2^p$, respectively.

The following statement is a part of classical invariant theory, see \cite[p. 318 in the 1893 original]{H}.
A modern proof can be found, e.g., in \cite[p. 65]{Sp}. We restate it for $S(V_p)^{SL_2(K)}$.

\begin{proposition}\label{transcendence degree of invariants}
The transcendence degree of the algebra of invariants $S(V_p)^{SL_2(K)}$ for the $SL_2(K)$-module $V_p$ is
\[
\text{\rm transc.deg}(S(V_p)^{SL_2(K)})=\begin{cases}
0, \text{ if } p=1;\\
1, \text{ if } p=2;\\
p-2, \text{ if } p>2.\\
\end{cases}
\]
\end{proposition}

As in the case of rational representations of $SL_2(K)$, every polynomial
representation of $GL_2(K)$ is a direct sum of irreducible representations. Every irreducible representation is of the form
${\det}^l\otimes \rho_p$, where $\det:GL_2(K)\to K$ is defined by $\det:g\to\det(g)$, $g\in GL_2(K)$. Clearly, the representation
${\det}^l\otimes \rho_p$ is the same as the representation indexed by the partition $(p+l,l)$.
If the $SL_2(K)$-module $KY_d$ is of the form
\[
KY_d=V_{p_1}\oplus\cdots\oplus V_{p_r},
\]
we extend the action of $SL_2(K)$ to the action of $GL_2(K)$ and then extend
it diagonally on the polynomial algebra $S(KY_d)$:
\[
g(f(Y_d))=g(f(y_1,\ldots,y_d))=f(g(y_1),\ldots,g(y_d)),\quad g\in GL_2(K),f(Y_d)\in S(KY_d).
\]
The homogeneous component $S(KY_d)_n$ of degree $n$ is a direct sum of
$GL_2(K)$-modules
\[
S(KY_d)_n=\sum_{l,p\geq 0}m_n(p,l)({\det}^l\otimes V_p),
\]
where the nonnegative integer $m_n(p,l)$ is the multiplicity of $\det^l\otimes V_p$ in $S(KY_d)_n$.

The following well known lemma, see e.g., \cite{DEP}, describes the $SL_2(K)$- and $UT_2(K)$-invariants of polynomial $GL_2(K)$-modules.

\begin{lemma}\label{invariants of SL2 and UT2}
Let the (finite dimensional) polynomial $GL_2(K)$-module $V$ decompose as
\[
V=\sum_{l,p\geq 0}m(p,l)({\det}^l\otimes V_p).
\]

{\rm (i)} The subspace $V^{SL_2(K)}$ of the $SL_2(K)$-invariants in $V$ is spanned by the one-dimensional $GL_2(K)$-submodules
${\det}^l\otimes V_0$.

{\rm (ii)} Each summand $m(p,l)({\det}^l\otimes V_p)$ of $V$ has $m(p,l)$ linearly independent elements $w_{l}^{(i)}$,
$i=1,2,\ldots,m(p,l)$, (one in each irreducible component ${\det}^l\otimes V_p$) such that the vector space $V^{UT_2(K)}$ has a basis
\[
\{w_{pl}^{(i)}\mid i=1,2,\ldots,m(p,l), p,l\geq 0\}.
\]
\end{lemma}

The above lemma implies that the homogeneous component
$(S(KY_d)^{SL_2(K)})_n$ of degree $n$ of the algebra of $SL_2(K)$-invariants $S(KY_d)^{SL_2(K)}$ is
\[
(S(KY_d)^{SL_2(K)})_n=\sum_{l\geq 0}m_n(0,l)({\det}^l\otimes V_0)
\]
and the Hilbert series of $S(KY_d)^{SL_2(K)}$ is
\[
H(S(KY_d)^{SL_2(K)},z)=\sum_{n\geq 0}\sum_{l\geq 0}m_n(0,l)z^n.
\]
Let $D_2(K)$ be the diagonal subgroup of $GL_2(K)$. Then ${\det}^l\otimes V_p$ has a basis $\{v_0,v_1,\ldots,v_p\}$ of eigenvectors
of $D_2(K)$ and
\[
\left(\begin{matrix}
\alpha_1&0\\
0&\alpha_2\\
\end{matrix}\right):v_j\to\alpha_1^{p+l-j}\alpha_2^{l+j}v_j,\quad 0\not=\alpha_1,\alpha_2\in K.
\]
The element $v_0$ is equal up to a multiplicative constant
to the element $w_{pl}\in {\det}^l\otimes V_p$ from Lemma \ref{invariants of SL2 and UT2} (ii).
The corresponding character of $D_2(K)$ is
\[
\chi({\det}^l\otimes V_p):\left(\begin{matrix}
\alpha_1&0\\
0&\alpha_2\\
\end{matrix}\right)\to S_{(p+l,l)}(\alpha_1,\alpha_2),
\]
where
\[
S_{(p+l,l)}(t_1,t_2)=(t_1t_2)^l(t_1^p+t_1^{p-1}t_2+\cdots+t_1t_2^{p-1}+t_2^p
\]
is the Schur function. This defines a ${\mathbb Z}^2$-grading (or bigrading) on ${\det}^l\otimes V_p$ and hence on $S(KY_d)$. The Hilbert series
of $S(KY_d)$ which counts both this bigrading and the usual grading is
\[
H_{GL_2(K)}(S(KY_d),t_1,t_2,z)=\sum_{l,p,n\geq 0}m_n(p,l)S_{(p+l,l)}(t_1,t_2)z^n.
\]
The Hilbert series
\[
H(S(KY_d),T_d)=H(S(KY_d),t_1,\ldots,t_d)=\prod_{i=1}^d\frac{1}{1-t_i}
\]
of the ${\mathbb Z}^d$-graded algebra $S(KY_d)$ and $H_{GL_2(K)}(S(KY_d),t_1,t_2,z)$ are related in the following way.
If $\{y_j,y_{j+1},\ldots,y_{j+p_i}\}$ is the basis of $V_{p_i}$, we replace the variables $t_j,t_{j+1},\ldots,t_{j+p_i}$, respectively,
with $t_1^{p_i}z,t_1^{p_i-1}t_2z,\ldots,t_2^{p_i}z$. Then
\[
H_{GL_2(K)}(S(KY_d),t_1,t_2,z)=H(S(KY_d),t_1^{p_1}z,t_1^{p_1-1}t_2z,\ldots,t_2^{p_r}z).
\]
There are many methods to compute the Hilbert series of $S(KY_d)^{SL_2(K)}$. We have chosen the following.
We define the multiplicity series of $H_{GL_2(K)}(S(KY_d),t_1,t_2,z)$
\[
S(KY_d),t_1,t_2,z)=\sum_{l,p,n\geq 0}m_n(p,l)t_1^p(t_1t_2)^lz^n,
\]
which is the generating function of the multiplicities $m_n(p,l)$ of the irreducible $GL_2(K)$-components in $S(KY_d)_n$.
We introduce new variables $t=t_1$, $u=t_1t_2$ and rewrite $M_{GL_2(K)}(S(KY_d),t_1,t_2,z)$ in the form
\[
M'_{GL_2(K)}(S(KY_d),t,u,z)=\sum_{l,p,n\geq 0}m_n(k,l)t^pu^lz^n.
\]
Then
\[
H(S(KY_d)^{SL_2(K)},z)=M'_{GL_2(K)}(S(KY_d),0,u,z)
\]
is the component of $M'_{GL_2(K)}(S(KY_d),t,u,z)$ which does not depend on $t$.
For the computing of $M_{GL_2(K)}(S(KY_d),t_1,t_2,z)$ and $M'_{GL_2(K)}(S(KY_d),t,u,z)$ we have used the methods in \cite{BBD}
which origin from classical work of Elliott \cite{E} and MacMahon \cite{MM}.
But it is much easier to verify that the formal power series $f(t_1,t_2,z)$ is equal to $M_{GL_2(K)}(S(KY_d),t_1,t_2,z)$.
It is sufficient to see whether
\[
H_{GL_2(K)}(S(KY_d),t_1,t_2,z)=\frac{1}{t_1-t_2}(t_1f(t_1,t_2,z)-t_2f(t_2,t_1,z)).
\]

\begin{remark}\label{how to recognize invariants}
Let $v$ be an element of the polynomial $GL_2(K)$-module $V$. In order to check that $v$ is $SL_2(K)$-invariant it is
not necessary to check that $g(v)=v$ for all $g\in SL_2(K)$. It is well known that it is enough to see this only for the matrices
$g_1$ and $g_2$ from Example \ref{action of g1 and g2}.
The matrices $g_1$ and $g_2$ are unitriangular. Hence the matrices $g_i-1$, $i=1,2$, act as nilpotent operators on $V$ and the logarithm
\[
\delta_i=\log(g_i)=\sum_{n\geq 1}(-1)^{n-1}\frac{(g_i-1)^n}{n},\quad i=1,2,
\]
is well defined.
One may extend the action of $\delta_i$ on $V$ to a derivation of $S(V)$ by the rule
$\delta(uv)=\delta(u)v+u\delta(v)$, $u,v\in S(V)$, and then by linearity on the whole $S(V)$.
Then $f\in S(V)$ is an $SL_2(K)$-invariant if and only if $\delta_i(f)=0$, $i=1,2$. This verification is much simpler than the checking that
$g_i(f)=f$, $i=1,2$. Similar arguments work for the tensor algebra on $V$ (which is isomorphic to the free associative algebra $K\langle V\rangle$),
for the free Lie algebra on $V$, etc.
\end{remark}

\begin{example}\label{invariants of second degree}
Consider the element
\[
w=\sum_{i=0}^p\binom{p}{i}(-1)^ie_1^{p-j}e_2^j\otimes e_1^je_2^{p-i}\in V_p\otimes V_p.
\]
This element is bihomogeneous of degree $(p,p)$.
Direct verification shows that $g_1(w)=g_2(w)=w$ (or, equivalently, $w$ belongs to the kernel of the derivations $\log(g_1)$
and $\log(g_2)$ of the tensor algebra on $V_p$). Hence
$w$ belongs to $(V_p\otimes V_p)^{SL_2(K)}$.
Hence $w$ generates an irreducible $GL_2(K)$-submodule of $V_p\otimes V_p$ isomorphic to $\det^p\otimes V_0$.
The explicit form of $w$ shows immediately that $w$ belongs to the direct summand $S^2(V_p)$ for $p$ even
and to the summand $\Lambda^2(V_p)$ of $V_p\otimes V_p$ for $p$ is odd.

Identifying $e_1^p,e_1^{p-1}e_2,\ldots,e_2^p\in V_p$ with $y_1,y_2,\ldots,y_{p+1}\in KY_{p+1}$, respectively,
we identify $V_p$ and $KY_{p+1}$. Then for $p=2q$ even the element $w\in S^2(V_p)$ becomes
\[
w_{p+1}=\sum_{i=1}^{p+1}\binom{p}{i-1}(-1)^{i-1}y_iy_{p+2-i}
\]
\[
=2\sum_{i=1}^q\binom{p}{i-1}(-1)^{i-1}y_iy_{p+2-i}+\binom{p}{q}(-1)^qy_{q+1}^2\in S(KY_{p+1})^{SL_2(K)}.
\]
Similarly, for $p=2q+1$ odd we identify $\Lambda^2(V_p)$ and the $GL_2(K)$-module $[KY_{p+1},KY_{p+1}]$
and obtain the $SL_2(K)$-invariants

\[
w_{q+1}=\sum_{i=1}^{p+1}\binom{p}{i-1}(-1)^{i-1}[y_i,y_{p+2-i}]
\]
\[
=2\sum_{i=1}^{q+1}\binom{p}{i-1}(-1)^{i-1}[y_i,y_{p+2-i}]\in (F_{p+1}')^{SL_2(K)}.
\]
\end{example}

\subsection{Noncommutative invariant theory for arbitrary groups.}
Let $\mathfrak V$ be a variety of $K$-algebras. We shall assume that $\mathfrak V$ is a variety of Lie algebras.
Since the base field $K$ is of characteristic 0, and hence infinite, the relatively free algebra $F_d({\mathfrak V})$
is ${\mathbb Z}^d$-graded. The action of the group $GL_d(K)=GL(KY_d)$ on $KY_d$ is extended diagonally on $F_d({\mathfrak V})$:
\[
g(f(Y_d))=g(f(y_1,\ldots,y_d))=f(g(y_1),\ldots,g(y_d)),\quad g\in GL_d(K),\, f(Y_d)\in F_d({\mathfrak V}).
\]
Let $G$ be a subgroup of $GL_d(K)=GL(KY_d)$. Then the algebra of $G$-invariants is
\[
F_d({\mathfrak V})^G=\{f(Y_d)\in F_d({\mathfrak V})\mid g(f(Y_d))=f(Y_d) \text{ for all } g\in D\}.
\]

The following lemma is well known, see e.g., \cite[Proposition 4.2]{DG} for the case of varieties of associative algebras.

\begin{lemma}\label{lifting invariants}
Let $\mathfrak V$ and $\mathfrak W$ be two varieties of algebras and ${\mathfrak V}\subset {\mathfrak W}$.
Let $G$ be a subgroup of $GL_d(K)=GL(KY_d)$.
Then the natural homomorphism $F_d({\mathfrak W})\to F_d({\mathfrak V})$ maps $F_d({\mathfrak W})^G$ onto $F_d({\mathfrak V})^G$.
Hence, if the algebra $F_d({\mathfrak V})^G$ is not finitely generated, then the same is $F_d({\mathfrak W})^G$.
\end{lemma}

As a result of the dichotomy of the varieties of Lie algebras (either $\mathfrak V$ contains ${\mathfrak A}^2$
or $\mathfrak V$ satisfies the Engel identity and is nilpotent), the $G$-invariants of the free metabelian Lie algebra are the key object
in the study of the finite generation of $F_d({\mathfrak V})^G$. We give a couple of cases when the algebra
$F_d({\mathfrak A}^2)^G$ is not finitely generated. We shall assume that $GL_d(K)$ acts in the same way on the vector spaces
$KY_d$, $KA_d$, and $KU_d$. If $g\in GL_d(K)$, then
\[
g(y_j)=\sum_{i=1}^d\alpha_{ij}y_i,\quad g(a_j)=\sum_{i=1}^d\alpha_{ij}a_i,\quad
g(u_j)=\sum_{i=1}^d\alpha_{ij}u_i,\quad j=1,\ldots,d,
\]
with the same $\alpha_{ij}$ for $Y_d$, $A_d$, and $U_d$.

\begin{lemma}\label{the invariants are not finitely generated - 1}
Let the subgroup $G$ of $GL_d(K)=GL(KY_d)$ be such that
the vector space $KY_d$ has no $G$-fixed elements (i.e., $KY_d^G=0$)
and the algebras $S(KY_d)^G$ and $F_d({\mathfrak A}^2)^G$ are nontrivial (i.e., $S(KY_d)^G\not=K$ and $F_d({\mathfrak A}^2)^G\not=0$).
Then the Lie algebra $F_d({\mathfrak A}^2)^G$ is not finitely generated.
\end{lemma}

\begin{proof}
Since the algebra $F_d^G=F_d({\mathfrak A}^2)^G$ is graded, the condition $KY_d^G=0$ gives that $F_d^G\subset F_d'$.
Hence the algebra $F_d^G$ is abelian and it is sufficient to show that it is infinite dimensional.
Working in the Poisson algebra $R_d$ from Subsection \ref{subsection on Poisson algebras}, if $w\in R_d'=[R_d,R_d]$ and
$f\in S(KU_d)\subset R_d$ are homogeneous nonzero elements, then the element $wf$ is also nonzero. Starting with a nonzero
homogeneous $w\in (F_d')^G$
and with a homogeneous polynomial of positive degree $f\in S(KU_d)^G\cong S(KY_d)^G$, we obtain an infinite sequence of elements
$wf^n\in F_d'$, $n=0,1,2,\ldots$ which are linearly independent in $R_d$ because are of different degrees. Since
$g(wf^n)=g(w)g(f)^n=wf^n$ for all $g\in G$, we conclude that $wf^n\in (F_d')^G$ and complete the proof.
\end{proof}

\begin{proposition}\label{the invariants are not finitely generated - 2}
Let $KY_d^G=0$ for the subgroup $G$ of $GL_d(K)=GL(KY_d)$ and let the transcendence degree of $S(KY_d)^G$ be greater than $1$. Then
the Lie algebra $F_d({\mathfrak A}^2)^G$ is not finitely generated.
\end{proposition}

\begin{proof}
Since $\text{transc.deg}(S(KY_d)^G)>1$, the algebra contains two algebraically independent homogeneous invariants $f_1$ and $f_2$.
Then the element
\[
\pi(f_1,f_2)=[f_1(A_d+U_d),f_2(A_d+U_d)]\in R_d
\]
belongs to $F_d'$ and is nonzero by Proposition \ref{independent polynomials in R_d}. Also, if $g\in G$, then
\[
g(f_1(A_d+U_d))=f_1(A_d+U_d), \quad g(f_2(A_d+U_d))=f_2(A_d+U_d).
\]
Hence
\[
g(\pi(f_1,f_2))=[g(f_1(A_d+U_d)),g(f_2(A_d+U_d))]=[f_1(A_d+U_d),f_2(A_d+U_d)]
\]
and $\pi(f_1,f_2)\in (F_d')^G$. Now the proof follows from Lemma \ref{the invariants are not finitely generated - 1}.
\end{proof}

Finally, computing the Hilbert series of the $SL_2(K)$-invariants of the free metabelian Lie algebra $F_d$, we may apply
exactly the same methods as for the $SL_2(K)$-invariants of $S(KY_d)$.
Starting with the decomposition of the $SL_2(K)$-module $KY_d$ and the Hilbert series $H(F_d,T_d)$ of $F_d$ we obtain the Hilbert series
\[
H_{GL_2(K)}(F_d,t_1,t_2,z)=\sum_{k,p,n\geq 0}m_{n,p,l}(F_d)S_{(p+l,l)}(t_1,t_2)z^n,
\]
the multiplicity series
\[
M_{GL_2(K)}(F_d,t_1,t_2,z)=\sum_{p,l,n\geq 0}m_{n,p,l}(F_d)t_1^p(t_1t_2)^lz^n,
\]
\[
M'_{GL_2(K)}(F_d,t,u,z)=\sum_{p,l,n\geq 0}m_{n,p,l}(F_d)t^pu^lz^n,
\]
and the Hilbert series
\[
H(F_d^{SL_2(K)},z)=M'_{GL_2(K)}(F_d,0,u,z)
\]
of $F_d^{SL_2(K)}$. See \cite{DDF} where such computations were performed for the multiplicity series of $F_d$.

\section{Infinite generation of the metabelian invariants\\
of the special linear group}

In this section we assume that $SL_2(K)$ acts rationally on the vector space $KY_d$ and
\[
KY_d\cong V_{p_1}\oplus\cdots\oplus V_{p_r},
\]
where $V_p\cong S^p(K^2)$ as an $SL_2(K)$-module.
When necessary we extend the $SL_2(K)$-action to $GL_2(K)$.
If $V'$ is an $SL_2(K)$-submodule of $KY_d$ we assume that it is spanned by a subset of $Y_d$.
We denote by $F_d$ the free metabelian Lie algebra generated by $Y_d$.

Following our agreement, if $KY_d$ has an $SL_2(K)$-invariant subspace we may assume that it has a basis $\{y_{m+1},\ldots,y_d\}$
and its complement has a basis $\{y_1,\ldots,y_m\}$. Then the subalgebra of $SL_2(K)$-invariants in $S(KY_d)$ has the form
\[
S(KY_d)^{SL_2(K)}=S(S(KY_m)^{SL_2(K)}\{y_{m+1},\ldots,y_d\}),
\]
where $S(S(KY_m)^{SL_2(K)}\{y_{m+1},\ldots,y_d\})$ is the symmetric $S(KY_m)^{SL_2(K)}$-algebra
on the free $S(KY_m)^{SL_2(K)}$-module generated by $\{y_{m+1},\ldots,y_d\}$.
We shall derive an analogue of this fact for the algebra $F_d^{SL_2(K)}$ which shows how to obtain the generators of
$F_d^{SL_2(K)}$ if we know the generators of $F_{d-1}^{SL_2(K)}$. The proof is based on a similar statement for the algebra
of $UT_2(K)$-invariants from \cite{DDF}.
We shall work in the Poisson algebra
\[
R_d=P(A_d,U_d)/(A_d^2,[A_d,A_d],[U_d,U_d],[a_i,y_j]-a_iy_j\mid i,j=1,\ldots,d)
\]
from Lemma \ref{Poisson structure on K[A,Y]}. The lemma gives that the Lie algebra $(KA_d)\text{\rm wr}(KU_d)$
is isomorphic to the Lie subalgebra of $R_d$ generated by $A_d$ and $U_d$. Hence the free metabelian Lie algebra $F_d$
embedded in $R_d$ as the Lie subalgebra generated by
\[
Y_d=A_d+U_d=\{y_i=a_i+u_i\mid i=1,\ldots,d\}.
\]
Clearly, the unitary associative subalgebra of $R_d$ generated by $Y_d$ is isomorphic to the symmetric algebra $S(KY_d)$.
By Proposition \ref{independent polynomials in R_d}, if $f_1(Y_d)$ and $f_2(Y_d)$ are two polynomials in
$S(KY_d)\subset R_d$, then the commutator $\pi(f_1(Y_d),f_2(Y_d))=[f_1(A_d+U_d),f_2(A_d+U_d)]$
belongs to the commutator ideal $F_d'$ of  $F_d\subset R_d$.

\begin{proposition}\label{invariants for d from these from d-1}
Let $y_d\in KY_d$ be $SL_2(K)$-invariant and let $KY_d=KY_{d-1}\oplus Ky_d$ as $SL_2(K)$-modules.
Let $\{v_i\mid i\in I\}$ and $\{w_j\mid j\in J\}$ be, respectively, homogeneous bases
of $(F_{d-1}')^{SL_2(K)}$ and $\omega(S(KY_{d-1}))^{SL_2(K)}$ with respect to the usual $\mathbb Z$-grading,
where $\omega(S(KY_{d-1}))$ is the augmentation ideal of $S(KY_{d-1})$. Then $(F_d')^{SL_2(K)}\subset R_d$ has a basis consisting of
\[
v_iu_d^n=v_i\text{\rm ad}^ny_d,\quad i\in I,\quad n\geq 0,
\]
\[
\pi(y_d,w_j)u_d^n=\sum_{l=1}^{d-1}[y_d,y_l]
\frac{\partial w_j}{\partial x_l}(\text{\rm ad}Y_{d-1})\text{\rm ad}^ny_d,\quad j\in J,\quad n\geq 0.
\]
\end{proposition}

\begin{proof}
We consider the group $UT_2(K)$ as a subgroup of $SL_2(K)$, both groups acting on $KY_d$. Hence
$UT_2(K)$ acts unitriangularly on the irreducible $SL_2(K)$-components of $KY_d$. Then the algebras of $UT_2(K)$-invariants
$S(KY_d)^{UT_2(K)}$ and $F_d^{UT_2(K)}$ coincide, respectively, with the algebras of constants (i.e., the kernels)
$S(KY_d)^{\delta_1}$ and $F_d^{\delta_1}$ of the Weitzenb\"ock (i.e., locally nilpotent linear) derivation
$\delta_1=\log(g_1)$. A basis of the vector space $(F_d')^{UT_2(K)}$ is given in \cite[Theorem 4.3]{DDF}.
We translate it in the language of our paper.
Let $\{v_i\mid i\in I\}$ and $\{w_j\mid j\in J\}$ be, respectively, homogeneous bases
of $(F_{d-1}')^{UT_2(K)}$ and $\omega(S(KY_{d-1}))^{UT_2(K)}$ with respect to both $\mathbb Z$- and ${\mathbb Z}^2$-gradings.
Then $(F_d')^{UT_2(K)}$ as a subspace of $(KA_d)\text{\rm wr}(KU_d)\subset R_d$ has a basis
\[
\{v_iu_d^n,\pi(y_d,w_j)u_d^n\mid i\in I,j\in J,n\geq 0\}.
\]
A bihomogeneous element in $S(KY_d)^{UT_2(K)}$ or in $F_d^{UT_2(K)}$ of bidegree $(p,q)$ is $SL_2(K)$-invariant if and only if $p=q$.
Since the elements $y_d,a_d,u_d$ are of bidegree $(0,0)$, the vector space $(F_d')^{SL_2(K)}$ has a basis consisting of the basis elements
$v_iu_d^n$ and $\pi(y_d,w_j)u_d^n$ of $(F_d')^{UT_2(K)}$ with the property that the bidegrees of $v_i$ and $w_j$
are equal to $(p_i,p_i)$ and $(q_j,q_j)$, respectively.
But this condition simply means that $v_i\in (F_{d-1}')^{SL_2(K)}$ and $w_j\in \omega(S(KY_{d-1}))^{SL_2(K)}$.
To complete the proof it is sufficient to see that
\[
v_iu_d^n=v_i\text{\rm ad}^ny_d,\quad i\in I,\quad n\geq 0,
\]
\[
\pi(y_d,w_j)y_d^n=\sum_{l=1}^{d-1}[y_d,y_l]
\frac{\partial w_j}{\partial y_l}(\text{\rm ad}Y_{d-1})\text{\rm ad}^ny_d\quad j\in J,\quad n\geq 0.
\]
\end{proof}

The proof of the following lemma is quite obvious and is similar to the proof of Lemma \ref{lifting invariants}.

\begin{lemma}\label{projection removing blocks}
Let $KY_d=V'\oplus V''$, where $V'=KY_e=K\{y_1,\ldots,y_e\}$ and $V''$ are $SL_2(K)$-submodules of $KY_d$
and let $F_e$ be the Lie subalgebra of $F_d$ generated by $Y_e$.
Then the natural projection $\nu:F_d\to F_e$ which send $y_{e+1},\ldots,y_d$ to $0$ maps $F_d^{SL_2(K)}$ onto $F_e^{SL_2(K)}$.
\end{lemma}

Our strategy to determine whether the algebra $F_d^{SL_2(K)}$ is finitely generated will be the following.
First we shall consider the case when $KY_{p+1}\cong V_p$ is an irreducible $SL_2(K)$-module
and shall describe the cases when $F_{p+1}^{SL_2(K)}$ is not finitely generated.
We shall show that this holds if and only if $p>2$.
In view of the above Lemma \ref{projection removing blocks}, this will imply that $F_d^{SL_2(K)}$ is never finitely generated if
$KY_d$ contains a submodule $V_p$, $p>2$. By similar arguments we shall eliminate the cases when
$KY_d$ contains a submodule isomorphic to $V_2\oplus V_p$, $p=0,1,2$, and $V_1\oplus V_1$. In all other cases, namely when
$KY_d=V_1\oplus V_0\oplus\cdots\oplus V_0$ or $d=3$ and $KY_d=V_2$, we shall show that $F_d^{SL_2(K)}$ is finitely generated.

\begin{lemma}\label{nonfinite generation case k at least 3}
Let $d=p+1$ and $KY_{p+1}\cong V_p$, $p\geq 3$.
Then the algebra $F_{p+1}^{SL_2(K)}$ is not finitely generated.
\end{lemma}

\begin{proof}
The group $SL_2(K)$ does not fix any nonzero element of $V_p$ because $p\geq 3$.
If $p\geq 4$ we use Proposition \ref{transcendence degree of invariants} which gives that the transcendence degree of
$S(KY_{p+1}]^{SL_2(K)}$ is equal to $p-2\geq 2$ and by Proposition \ref{the invariants are not finitely generated - 2}
the algebra $F_{p+1}^{SL_2(K)}$ is not finitely generated.
If $p=3$ we shall apply Lemma \ref{the invariants are not finitely generated - 1}.
Since the transcendence degree of $S(KY_4)^{SL_2(K)}\not=K$ is equal to 1 and $S(KY_4)^{SL_2(K)}\not=K$,
it is sufficient to show that $F_4^{SL_2(K)}\not=0$. A nonzero element in $F_4^{SL_2(K)}$ is given in
Example \ref{invariants of second degree}.
\end{proof}

\begin{remark}
We may give a constructive proof of Lemma \ref{nonfinite generation case k at least 3}.
Example \ref{invariants of second degree} provides a nonzero element in $F_{p+1}^{SL_2(K)}$ for $p\geq 3$ odd.
In order to apply Lemma \ref{the invariants are not finitely generated - 1}
we need an $SL_2(K)$-invariant element of positive degree in $S(KY_{p+1})$.
For this purpose we may take the discriminant.
Recall that if $\theta_1,\ldots,\theta_p$ are the zeros of the polynomial
$f(t)=\xi_0t^p+\xi_1t^{p-1}+\cdots+\xi_{p-1}t+\xi_p$ with coefficients
$\xi_0,\xi_1,\ldots,\xi_p$ in $S(V_p)$, then the discriminant
\[
\xi_0^{2p-2}\prod_{1\leq i<j\leq p}(\theta_i-\theta_j)^2
\]
of $f(t)$ is an $SL_2(K)$-invariant element of degree $2(p-1)$ in $S(V_p)=S(KY_{p+1})$.
For $p>3$ even we may choose one more explicit $SL_2(K)$-invariant of second degree in $S(KY_{p+1})$, namely the element
in Example \ref{invariants of second degree}. Then we may produce a nonzero element in $F_{p+1}^{SL_2(K)}$ applying
Proposition \ref{independent polynomials in R_d}.
\end{remark}

\begin{lemma}\label{finite generation case 10000}
Let
\[
KY_d\cong V_1\oplus V_0\oplus\cdots\oplus V_0,
\]
where $V_1$ and $V_0\oplus\cdots\oplus V_0$ have bases
$Y_2$ and $\{y_3,\ldots,y_d\}$, respectively. Then the algebra $F_d^{SL_2(K)}$ is generated by $[y_2,y_1]$ and
$\{y_3,\ldots,y_d\}$.
\end{lemma}

\begin{proof}
We proceed by induction on $d$ applying Proposition \ref{invariants for d from these from d-1} in each step.
For $d=2$ we have that $S(KY_2)^{SL_2(K)}=S(KV_1)^{SL_2(K)}=K$. Since as a $GL_2(K)$-module $F_2'$
decomposes as a direct sum of $\det\otimes V_n$, $n\geq 0$, we obtain that $F_2^{SL_2}$ is the one-dimensional vector space
spanned by the commutator $[y_2,y_1]$. By induction, we assume that the Lie algebra $F_{d-1}^{SL_2(K)}$ is generated by
$[y_2,y_1],y_3,\ldots,y_{d-1}$ and has a basis consisting of all commutators
\[
[y_2,y_1]\text{ad}^{n_3}y_3\cdots\text{ad}^{n_{d-1}}y_{d-1},\quad n_i\geq 0,\quad i=3,\ldots,d-1,
\]
starting with $[y_2,y_1]$ and all commutators
\[
[y_{i_1},y_{i_2},\ldots,y_{i_n}],\quad i_1>i_2\leq\cdots\leq i_n,\quad i_j=3,\ldots,d-1,\quad n=1,2,\ldots .
\]
The vector space of the $SL_2(K)$-invariants in the augmentation ideal of $S(KY_{d-1})$
has a basis consisting of all monomials in $y_3,\ldots,y_{d-1}$.
Now Proposition \ref{invariants for d from these from d-1} implies that the $SL_2(K)$-invariants in $F_d'$ are spanned by
the elements $[y_2,y_1]\text{ad}^{n_3}y_3\cdots\text{ad}^{n_d}y_d$ and the commutators of length $\geq 2$ in $y_3,\ldots,y_d$.
Since $y_3,\ldots,y_d$ belong to $F_d^{SL_2(K)}$, this implies that $F_d^{SL_2(K)}$ is generated by
$[y_2,y_1]$ and $\{y_3,\ldots,y_d\}$.
\end{proof}

\begin{lemma}\label{finite generation case 2}
Let $d=3$ and $KY_3\cong V_2$,
Then $F_3^{SL_2(K)}=0$.
\end{lemma}

\begin{proof}
In \cite[Example 3.2]{DDF} we have computed the Hilbert series of the algebra $F_3^{UT_2(K)}$:
\[
H_{GL_2(K)}(F_3^{UT_2(K)},t_1,t_2,z)=t_1^2z+\frac{t_1^3t_2z^2}{(1-t_1^2z)(1-t_1t_2z)}
\]
\[
=t_1^2z+\frac{t_1(t_1t_2)^2z^2}{(1-t_1^2z)(1-t_1t_2z)}=t^2z+\frac{tu^2z^2}{(1-t^2z)(1-uz)},
\]
where $t=t_1$, $u=t_1t_2$. Since the Hilbert series of $F_3^{SL_2(K)}$ is obtained from the Hilbert series of
$F_3^{UT_2(K)}$ substituting $t=0$ and $u=1$, we obtain that $H(F_3^{SL_2(K)},z)=0$, and hence $F_3^{SL_2(K)}=0$.
\end{proof}

As we have agreed below Example \ref{action of g1 and g2}, if $V_p$ is an irreducible component of the $SL_2(K)$-module $KY_d$,
then $V_p$ has a basis $\{y_i,y_{i+1},\ldots,y_{i+p}\}$
with the same action of $SL_2(K)$ as on the basis $\{e_1^p,e_1^{p-1}y_2,\ldots,e_2^p\}$.

\begin{lemma}\label{nonfinite generation case 22, 31, 32, 33}
Let $KY_d\cong V_1\oplus V_1,V_2\oplus V_0,V_2\oplus V_1$ or $V_2\oplus V_2$.
Then $F_d^{SL_2(K)}$ is not finitely generated.
\end{lemma}

\begin{proof}
If $KY_d\cong V_i\oplus V_j$, we assume that $V_i$ and $V_j$ have bases $\{y_1,\ldots,y_{i+1}\}$ and
$\{y_{i+2},\ldots,y_{i+j+2}\}$, respectively.
First, we shall construct nonzero elements in $(F_d')^{SL_2(K)}$.

(i) Let $d=4$ and $KY_4\cong V_1\oplus V_1$. Consider the element
\[
v_{11}=[y_1,y_4]-[y_2,y_3]\in F_4.
\]
Applying Remark \ref{how to recognize invariants} we see that $g_1(v_{11})=g_2(v_{11})=v_{11}$.
Hence $v_{11}\in F_4^{SL_2(K)}$.

(ii) Let $d=4$ and $KY_4\cong V_2\oplus V_0$.
It is well known for more than 100 years, see \cite{G}, that
the element $f=y_2^2-y_1y_3$ belongs to $S(KY_3)^{SL_2(K)}=S(V_2)^{SL_2(K)}$.
Since $y_4$ is an $SL_2(K)$-invariant, applying Proposition \ref{invariants for d from these from d-1} we obtain the nonzero
element
\[
v_{20}=\pi(y_4,f)=\sum_{l=1}^{3}[y_4,y_l]
\frac{\partial f}{\partial y_l}(\text{\rm ad}Y_3)
\]
\[
=2[y_4,y_2,y_2]-[y_4,y_1,y_3]-[y_4,y_3,y_1]\in (F_4')^{SL_2(K)}.
\]

(iii) Let $d=5$ and $KY_5\cong V_2\oplus V_1$. Then the simplest nonzero $SL_2(K)$-invariant is
\[
v_{21}=[y_4,y_5]\in F_5^{SL_2(K)}.
\]

(iv) Let $d=6$ and $KY_6\cong V_2\oplus V_2$. Using the element
\[
w=\sum_{i=0}^2\binom{2}{i}(-1)^ie_1^{2-j}e_2^j\otimes e_1^je_2^{2-i}\in V_2\otimes V_2
\]
from Example \ref{invariants of second degree}
we can obtain the nonzero $SL_2(K)$-invariant
\[
v_{22}=[y_1,y_6]-2[y_2,y_5]+[y_3,y_4]\in F_6^{SL_2(K)}.
\]

It is also known that if $KY_4=V_1\oplus V_1$,
then $S(KY_4)^{SL_2(K)}=S(K(y_1y_4-y_2y_3))$.
Hence in all four cases we have $SL_2(K)$-invariants of positive degree in
$S(KY_d)$. In (i) this  is $f=y_1y_4-y_2y_3$ and in (ii)--(iv) -- $f=y_2^2-y_1y_3$.
Since $KY_d^{SL_2(K)}=0$ in the cases (i), (iii) and (iv),
by Lemma \ref{the invariants are not finitely generated - 1} we conclude that the algebra $F_d^{SL_2(K)}$ is not finitely generated.

Let us assume that in the case (ii) the algebra $F_4^{SL_2(K)}$ is finitely generated. Since $KY_4^{SL_2(K)}=Ky_4$, this implies
that $F_4^{SL_2(K)}$ is generated by $y_4$ and a finite number of elements $w_1,\ldots,w_m\in (F_4')^{SL_2(K)}$.
Hence $(F_4')^{SL_2(K)}$ is spanned by
\[
w_i\text{ad}^ny_4,\quad i=1,\ldots,m,\quad n\geq 0.
\]
But $(F_4')^{SL_2(K)}$ contains the infinite sequence of elements
\[
v_{22}(\text{ad}^2y_2-\text{ad}y_1\text{ad}y_3)^l,\quad l=0,1,2,\ldots,
\]
and for sufficiently large $l$ these elements cannot be expressed as linear combinations of $w_i\text{ad}^ny_4$.
Again, $F_4^{SL_2(K)}$ is not finitely generated.
\end{proof}

The following theorem is one of the main results of the paper.

\begin{theorem}\label{when the algebra is finitely generated}
Let $KY_d$ be isomorphic as an $SL_2(K)$-module to the direct sum $V_{p_1}\oplus\cdots\oplus V_{p_r}$,
where $(p_1,\ldots,p_r)\not=(0,\ldots,0)$
and let $F_d=F_d({\mathfrak A}^2)$ be the $d$-generated free metabelian Lie algebra. Then the algebra
$F_d^{SL_2(K)}$ of $SL_2(K)$-invariants is finitely generated if and only if
$KY_d\cong V_1\oplus V_0\oplus\cdots\oplus V_0$ or $KY_d\cong V_2$.
\end{theorem}

\begin{proof}
Let $KY_d\cong V_{p_1}\oplus\cdots\oplus V_{p_r}$. We may assume that $p_1\geq \cdots\geq p_r$.
By Lemmas \ref{finite generation case 10000} and \ref{finite generation case 2} the algebra $F_d^{SL_2(K)}$ is finitely generated
in the cases $KY_d\cong V_1\oplus V_0\oplus\cdots\oplus V_0$ and $KY_d\cong V_2$. Hence we have to show that in the cases
(i) $p_1\geq 3$; (ii) $p_1=2$, $r\geq 2$; (iii) $p_1=p_2=1$ the algebra $F_d^{SL_2(K)}$ is not finitely generated.
Applying Lemma \ref{projection removing blocks} we reduce the considerations to the following three cases:

$(\text{i}')$ $KY_d\cong V_p$, $p\geq 3$. By Lemma \ref{nonfinite generation case k at least 3} the algebra
$F_d^{SL_2(K)}$ is not finitely generated.

$(\text{ii}')$ $KY_d\cong V_2\oplus V_i$, $i=0,1,2$. Now we apply Lemma \ref{nonfinite generation case 22, 31, 32, 33}.

$(\text{iii}')$ $KY_d\cong V_1\oplus V_1$. We apply Lemma \ref{nonfinite generation case 22, 31, 32, 33} again and complete
the proof of the theorem.
\end{proof}

\section{Explicit invariants in small dimensions}

In this section we compute the explicit generators of the finitely generated $S(KY_d)^{SL_2(K)}$-module $(F_d')^{SL_2(K)}$
for $d\leq 5$ and for $d=6$ when $KY_6\cong V_2\oplus V_2$ and $KY_6\cong V_1\oplus V_1\oplus V_1$.
Applying Proposition \ref{invariants for d from these from d-1} we omit the cases containing $V_0$ as a component of $KY_d$.
In what follows, the Hilbert series are obtained from \cite[Example 3.2]{DDF} and the generators and the relations
between them from \cite[Section 5]{DDF}, taking into account the homogeneous components of bidegree $(n,n)$ only.

\begin{example}\label{only the Hilbert series, generators, and relations}

$(\text{i})$ Let $d=3$ and $KY_3\cong V_2$. In this case we have from Lemma \ref{finite generation case 2} that
\[
H((F_3')^{SL_2(K)},z)=0,
\]
and the algebra $(F_3')^{SL_2(K)}$ does not have nonzero elements.

$(\text{ii})$ Let $d=4$ and $KY_4\cong V_3$. Then
\[
H((F_4')^{SL_2(K)},z)=\frac{z^2}{1-z^4} \quad \text{and}  \quad H(S(KY_4)^{SL_2(K)},z)=\frac{1}{1-z^4}.
\]
The $S(KY_4)^{SL_2(K)}$-module $(F_4')^{SL_2(K)}$ has one generator of degree $2$
\[
v=[y_4,y_1]-3[y_3,y_2],
\]
and the single generator of $S(KY_4)^{SL_2(K)}$ is
\[
f=y_1^2y_4^2-6y_1y_2y_3y_4+4y_1y_3^3+4y_2^3y_4-3y_2^2y_3^2.
\]

$(\text{iii})$ Let $d=4$ and $KY_4\cong V_1\oplus V_1$.
By Lemma \ref{nonfinite generation case 22, 31, 32, 33} we have a nonzero ${SL_2(K)}$-invariant in $(F_4')^{SL_2(K)}$.
We shall find all generators of the
finitely generated $S(KY_4)^{SL_2(K)}$-module $(F_4')^{SL_2(K)}$. Using the suggested above method we have
\[
H((F_4')^{SL_2(K)},z)=\frac{3z^2}{1-z^2} \quad \text{and}  \quad H(S(KY_4)^{SL_2(K)},z)=\frac{1}{1-z^2}.
\]
The free generators of the $S(KY_4)^{SL_2(K)}$-module $(F_4')^{SL_2(K)}$ are
\[
v_1=[y_4,y_1]-[y_3,y_2], \quad v_2=[y_2,y_1] \quad \text{and} \quad v_3=[y_4,y_3].
\]
The single generator of $S(KY_4)^{SL_2(K)}$ is $f=y_1y_4-y_2y_3$.

$(\text{iv})$ Let $d=5$ and $KY_5\cong V_4$. Then
\[
H((F_5')^{SL_2(K)},z)=\frac{z^5}{(1-z^2)(1-z^3)},
\quad
H(K[X_5]^{SL_2(K)},z)=\frac{1}{(1-z^2)(1-z^3)}.
\]
The only generator of the $S(KY_5)^{SL_2(K)}$-module $(F_5')^{SL_2(K)}$ is
\[
v=[y_2,y_1]\cdot(12y_4^3+4y_2y_5^2-16y_3y_4y_5)+
[y_3,y_1]\cdot(-2y_1y_5^2+18y_3^2y_5-18y_3y_4^2+4y_2y_4y_5)
\]
\[
+[y_3,y_2]\cdot(16y_2y_4^2-18y_2y_3y_5)+[y_4,y_1]\cdot(-4y_2y_4^2+12y_3^2y_4+4y_1y_4y_5-16y_2y_3y_5)
\]
\[
+[y_4,y_2]\cdot(-8y_2y_3y_4+12y_2^2y_5)+[y_5,y_1]\cdot(y_1y_3y_5-9y_3^3-3y_1y_4^2-y_2^2y_5+14y_2y_3y_4)
\]
\[
+[y_5,y_2]\cdot(-8y_2^2y_4+6y_2y_3^2),
\]
where $w\cdot p(Y_m)=wp(\text{ad}Y_m)$ for $w\in F_5'$ and $p(Y_m)\in S(KY_m)$. The generators of $S(KY_5)^{SL_2(K)}$ are
\[
f_1=y_1y_5-4y_2y_4+3y_3^2,\quad
%\]
%\[
f_2=-y_1y_3y_5-2y_2y_3y_4+y_3^3+y_1y_4^2+y_2^2y_5.
\]

$(\text{v})$ Let $d=5$ and $KY_5\cong V_2\oplus V_1$. Then
\[
H((F_5')^{SL_2(K)},z)=\frac{z^2+z^3+z^4+z^5}{(1-z^2)(1-z^3)},
%\]
%and
%\[
\quad
H(S(KY_5)^{SL_2(K)},z)=\frac{1}{(1-z^2)(1-z^3)}.
\]
The four generators of the $S(KY_5)^{SL_2(K)}$-module $(F_5')^{SL_2(K)}$ are
\[
v_1=[y_5,y_4],
\]
\[
v_2=[y_4,y_2]\cdot y_5-[y_4,y_3]\cdot y_4-[y_5,y_1]\cdot y_5+[y_5,y_2]\cdot y_4,
\]
\[
v_3=[y_4,y_1]\cdot y_3y_5-[y_4,y_2]\cdot y_3y_4+[y_4,y_3]\cdot (y_2y_4-y_1y_5)-[y_5,y_1]\cdot y_2y_5+[y_5,y_2]\cdot y_1y_5,
\]
\[
v_4=-[y_4,y_1]\cdot y_3^2y_4+2[y_4,y_2]\cdot y_2y_3y_4-[y_4,y_3]\cdot y_1y_3y_4+[y_5,y_1]\cdot (2y_2y_3y_4-y_1y_3y_5)
\]
\[
+[y_5,y_2]\cdot (2y_1y_2y_5-4y_2^2y_4)+[y_5,y_3]\cdot (2y_1y_2y_4-y_1^2y_5)+[y_5,y_4]\cdot (2y_2^3-2y_1y_2y_3).
\]
Note that the first generator was suggested in Lemma \ref{nonfinite generation case 22, 31, 32, 33}.
The generators of $S(KY_5)^{SL_2(K)}$ are
\[
f_1=y_2^2-y_1y_3, \quad f_2=y_1y_5^2-2y_2y_4y_5+y_3y_4^2.
\]

$(\text{vi})$ Let $d=6$ and $KY_6\cong V_1\oplus V_1\oplus V_1$. Then
\[
H((F_6')^{SL_2(K)},z)=\frac{6z^2-z^6}{(1-z^2)^3},
\]
and
\[
H(S(KY_6)^{SL_2(K)},z)=\frac{1}{(1-z^2)^3}.
\]
The generators of the $S(KY_6)^{SL_2(K)}$-module $(F_6')^{SL_2(K)}$ are
\[
v_1=[y_2,y_1], \quad v_2=[y_4,y_3], \quad v_3=[y_6,y_5],\quad v_4=[y_4,y_1]-[y_3,y_2],
\]
\[
v_5=[y_6,y_1]-[y_5,y_2],\quad v_6=[y_6,y_3]-[y_5,y_4].
\]
The generators of $S(KY_6)^{SL_2(K)}$ are
\[
f_1=y_1y_4-y_2y_3, \quad f_2=y_1y_6-y_2y_5, \quad f_3=y_3y_6-y_4y_5.
\]
There is a relation of degree $6$ between the generators $v_1,\ldots,v_6$:
\[
v_1\cdot f_3^2+v_2\cdot f_2^2+v_3\cdot f_1^2-v_4\cdot f_2f_3+v_5\cdot f_1f_3-v_6\cdot f_1f_2=0.
\]

$(\text{vii})$ Let $d=6$ and $KY_6\cong V_2\oplus V_2$. Then
\[
H((F_6')^{SL_2(K)},z)=\frac{z^2+2z^3+3z^4-z^6}{(1-z^2)^3},
%\]
\quad
%\[
H(S(KY_6)^{SL_2(K)},z)=\frac{1}{(1-z^2)^3}.
\]
The $K[X_6]^{SL_2(K)}$-module $(F_6')^{SL_2(K)}$ has 6 generators.
The first one was obtained in Lemma \ref{nonfinite generation case 22, 31, 32, 33}:
\[
v_1=[y_6,y_1]-2[y_5,y_2]+[y_4,y_3],
\]
\[
v_2=[y_2,y_1]\cdot y_6-[y_3,y_1]\cdot y_5+[y_3,y_2]\cdot y_4,
\]
\[
v_3=[y_6,y_1]\cdot y_5-[y_5,y_1]\cdot y_6+[y_4,y_2]\cdot y_6-[y_4,y_3]\cdot y_5-[y_6,y_2]\cdot y_4+[y_5,y_3]\cdot y_4,
\]
\[
v_4=[y_3,y_1]\cdot y_5^2+[y_6,y_1]\cdot (2y_2y_5-y_1y_6)-[y_4,y_3]\cdot y_3y_4-2[y_2,y_1]\cdot y_5y_6
\]
\[
+[y_5,y_1]\cdot (2y_2y_6-2y_3y_5)+[y_4,y_2]\cdot (2y_3y_5-y_2y_6)
\]
\[
+[y_5,y_2]\cdot (2y_3y_4-2y_2y_5)-2[y_3,y_2]\cdot y_4y_5-[y_6,y_2]\cdot y_2y_4,
\]
\[
v_5=[y_3,y_1]\cdot (y_1y_6-2y_2y_5)+[y_6,y_1]\cdot (y_2^2-y_1y_3)+[y_2,y_1]\cdot (2y_3y_5-2y_2y_6)
\]
\[
+2[y_5,y_1]\cdot y_2y_3+[y_4,y_2]\cdot y_2y_3-2[y_5,y_2]\cdot y_2^2+[y_3,y_2]\cdot y_2y_4-[y_4,y_1]\cdot y_3^2,
\]
\[
v_6=[y_6,y_1]\cdot y_4y_6+[y_4,y_3]\cdot y_4y_6-2[y_5,y_1]\cdot y_5y_6-2[y_4,y_2]\cdot y_5y_6+4[y_5,y_2]\cdot y_5^2
\]
\[
-2[y_6,y_2]\cdot y_4y_5-2[y_5,y_3]\cdot y_4y_5+[y_4,y_1]\cdot y_6^2+[y_6,y_3]\cdot y_4^2,
\]
while the generators of $K[Y_6]^{SL_2(K)}$ are
\[
f_1=y_1y_3-y_2^2, \quad f_2=y_1y_6-2y_2y_5+y_3y_4, \quad f_3=y_4y_6-y_5^2,
\]
with a relation between the generators of $(F_6')^{SL_2(K)}$:
\[
v_4\cdot f_2+v_5\cdot f_3-v_6\cdot f_1+v_1\cdot (f_2^2+f_1f_3)=0.
\]
\end{example}

Now we explain one of the cases in detail.

\begin{example}
Let $d=4$ and $KY_4\cong V_3$. The Hilbert series of the algebras $(F_4')^{UT_2(K)}$ and $S(KY_4)^{UT_2(K)}$
are computed in \cite[Example 5.2]{DDF}:
\[
H_{GL_2(K)}((F_4')^{UT_2(K)},t_1,t_2,z)=\frac{t_1^3t_2z^2(t_1^2+t_2^2+t_1^4t_2^4z^2+t_1^5t_2^6z^3-t_1^8t_2^6z^4)}{(1-t_1^3z)(1-t_1^2t_2z)(1-t_1^6t_2^6z^4)}
\]
\[
=\frac{t^4uz^2+u^3z^2+t^2u^5z^4+tu^7z^5-t^4u^7z^6}{(1-t^3z)(1-tuz)(1-u^6z^4)},
\]
\[
H_{GL_2(K)}(S(KY_4)^{UT_2(K)},t_1,t_2,z)=\frac{1+t_1^6t_2^3z^3}{(1-t_1^3z)(1-t_1^4t_2^2z^2)(1-t_1^6t_2^6z^4)}
\]
\[
=\frac{1+t^3u^3z^3}{(1-t^3z)(1-t^2u^2z^2)(1-u^6z^4)},
\]
where $t=t_1$, $u=t_1t_2$. The Hilbert series of $(F_4')^{SL_2(K)}$ and $S(KY_4)^{SL_2(K)}$ are obtained from the above Hilbert series
 substituting $t=0$ and $u=1$. Note that the degree of $u$ gives the information of the symmetric bidegrees of the homogeneous generators. Thus we have
\[
H((F_4')^{SL_2(K)},z)=\frac{z^2}{1-z^4}, \quad H(S(KY_4)^{SL_2(K)},z)=\frac{1}{1-z^4}.
\]
The Hilbert series suggest that the $S(KY_4)^{SL_2(K)}$-module $(F_4')^{SL_2(K)}$ has one generator of bidegree $(3,3)$ and degree $2$,
and $S(KY_4)^{SL_2(K)}$ is generated by a generator of bidegree $(6,6)$ and degree $4$. A candidate for a generator of $S(KY_4)^{SL_2(K)}$ is of the form
\[
f=\alpha_1y_1^2y_4^2+\alpha_2y_1y_2y_3y_4+\alpha_3y_1y_3^3+\alpha_4y_2^3y_4+\alpha_5y_2^2y_3^2.
\]
Since it is ${SL_2(K)}$-invariant, it is preserved under the actions of $g_1$ and $g_2$.
Easy computations give that $(\alpha_1,\alpha_2,\alpha_3,\alpha_4,\alpha_5)=(1,-6,4,4,-3)$.
Using similar arguments we obtain the generator of  the $S(KY_4)^{SL_2(K)}$-module $(F_4')^{SL_2(K)}$ of the form
\[
v=\beta_1[y_4,y_1]+\beta_2[y_3,y_2],
\]
and find that $(\beta_1,\beta_2)=(1,-3)$.

On the other hand, the Hilbert series of $S(KY_4)^{SL_2(K)}$-submodule of $(F_4')^{SL_2(K)}$ generated by $v=[y_4,y_1]-3[y_3,y_2]$
is equal to the Hilbert series of whole algebra. Thus $(F_4')^{SL_2(K)}$ is a free cyclic $S(KY_4)^{SL_2(K)}$-module generated by $v$. As a vector space
$(F_4')^{SL_2(K)}$ is spanned by the elements $v\cdot f^n$, $n\geq 0$, forming an infinite set of generators of $(F_4')^{SL_2(K)}$ as an algebra.
\end{example}

\section*{Acknowledgements}

The authors are deeply obliged to the anonymous referee for the careful reading and the numerous valuable constructive suggestions for improving
both the mathematics and the exposition of the paper.
The second named author is very thankful to the Institute of Mathematics and Informatics of
the Bulgarian Academy of Sciences for the creative atmosphere and the warm hospitality during his visit
as a post-doctoral fellow when this project was carried out.
The first named author is grateful to Leonid Bedratyuk for the useful discussions on classical invariant theory.


\begin{thebibliography}{ABC}

\bibitem[ADF]{ADF}
G. Almkvist, W. Dicks, E. Formanek,
Hilbert series of fixed free algebras and noncommutative classical
invariant theory,
J. Algebra {\bf 93} (1985), 189-214.

\bibitem[Ba]{Ba}
Yu. A. Bahturin,
Identical Relations in Lie Algebras (Russian),
Nauka, Moscow, 1985.
Translation: VNU Science Press, Utrecht, 1987.

\bibitem[BBD]{BBD}
F. Benanti, S. Boumova, V. Drensky, G.K. Genov, P. Koev,
Computing with rational symmetric functions and applications to
invariant theory and PI-algebras,
Serdica Math. J. {\bf 38} (2012), 137-188.
http://www.math.bas.bg/serdica/2012/2012-137-188.pdf.

\bibitem[Br]{Br}
R.M. Bryant,
On the fixed points of a finite group acting on a free Lie algebra,
J. London Math. Soc. (2) {\bf 43} (1991), 215-224.

\bibitem[DDF]{DDF}
R. Dangovski, V. Drensky, \c{S}. F{\i}nd{\i}k,
Weitzenb\"ock derivations of free metabelian Lie algebras,
Linear Algebra Appl. {\bf 439} (2013), No. 10, 3279-3296.

\bibitem[DEP]{DEP}
C. De Concini, D. Eisenbud, C. Procesi,
Young diagrams and determinantal varieties,
Invent. Math. {\bf 56} (1980), 129-165.

\bibitem[DD]{DD}
M. Domokos, V. Drensky,
A Hilbert-Nagata theorem in noncommutative invariant theory,
Trans. Amer. Math. Soc. {\bf 350} (1998), 2797-2811.

\bibitem[D1]{D1}
V. Drensky,
Fixed algebras of residually nilpotent Lie algebras,
Proc. Amer. Math. Soc. {\bf 120} (1994), 1021-1028.

\bibitem[D2]{D2}
V. Drensky,
Commutative and noncommutative invariant theory,
Math. and Education in Math.,
Proc. of the 24-th Spring Conf. of the Union of Bulgar. Mathematicians,
Svishtov, April 4-7, 1995, Sofia, 1995, 14-50.

\bibitem[DG]{DG}
V. Drensky, C.K. Gupta,
Constants of Weitzenb\"ock derivations and invariants of unipotent
transformations acting on relatively free algebras,
J. Algebra {\bf 292} (2005), 393-428.

\bibitem[ER]{ER}
R. Ehrenborg, G.-C. Rota,
Apolarity and canonical forms for homogeneous polynomials,
Eur. J. Comb. {\bf 14} (1993), 157-181.

\bibitem[E]{E}
E.B. Elliott,
On linear homogeneous diophantine equations,
Quart. J. Pure Appl. Math. {\bf 34} (1903), 348-377.

\bibitem[F]{F}
E. Formanek,
Noncommutative invariant theory,
Contemp. Math. {\bf 43} (1985), 87-119.

\bibitem[G]{G}
P. Gordan,
Vorlesungen \"uber Invariantentheorie, Herausgegeben von G. Kerschensteiner, Zweiter Band: Binare Formen,
Teubner, Leipzig, 1887. Reprinted by Chelsea Publishing Company, New York, 1987.

\bibitem[H]{H}
D. Hilbert,
\"Uber die vollen Invariantensysteme,
Math. Ann. {\bf 42} (1893), 313-373.
Reprinted in Gesammelte Abhandlungen, Band II: Algebra. Invariantentheorie. Geometrie, 282-344,
Springer-Verlag, Berlin-Heidelberg-New York, 1970.

\bibitem[J]{J}
C.G.J. Jacobi,
De determinantibus functionalibus,
J. Reine Angew. Math. {\bf 22} (1841), 319-359.
Translation: Ueber die Functionaldeterminanten,
Ostwald's Klassiker d. exacten Wissensch. No. 78. Leipzig, W. Engelmann, 1896.

\bibitem[KS]{KS}
O.G. Kharlampovich, M.V. Sapir,
Algorithmic problems in varieties,
Intern. J. Algebra and Computation {\bf 5} (1995), 379-602.

\bibitem[Ko]{Ko}
A.I. Kostrikin,
The Burnside problem,
Izv. Akad. Nauk SSSR, Ser. Mat. {\bf 23} (1959), 3-34 (Russian).
Translation: Amer. Math. Soc. Transl. (2) {\bf 36} (1964), 63-99.

\bibitem[Md]{Md}
I.G. Macdonald,
Symmetric Functions and Hall Polynomials,
Oxford Univ. Press (Clarendon), Oxford, 1979, Second Edition, 1995.

\bibitem[MM]{MM}
P.A. MacMahon,
Combinatory Analysis, vols. 1 and 2,
Cambridge Univ. Press. 1915, 1916.
Reprinted in one volume:
Chelsea, New York, 1960.

\bibitem[MLU]{MLU}
L. Makar-Limanov, U. Umirbaev,
The Freiheitssatz for Poisson algebras,
J. Algebra {\bf 328} (2011), 495-503.

\bibitem[SU]{SU}
I.P. Shestakov, U.U. Umirbaev,
Poisson brackets and two-generated subalgebras of rings of polynomials,
J. Amer. Math. Soc. {\bf 17} (2004), 181-196.

\bibitem[Sh]{Sh}
A.L. Shmel'kin,
Wreath products of Lie algebras and
their application in the theory of groups,
Trudy Moskov. Mat. Obshch. {\bf 29} (1973), 247-260  (Russian).
Translation: Trans. Moscow Math. Soc. {\bf 29} (1973), 239-252.

\bibitem[Sp]{Sp}
T.A. Springer,
Invariant Theory,
Lecture Notes in Mathematics, {\bf 585}, Springer-Verlag, Berlin-Heidelberg-New York, 1977.

%\bibitem[T]{T}
%R.M. Thrall,
%On symmetrized Kronecker powers and the structure of the free Lie ring,
%Amer. J. Math. 64 (1942), 371-388.

\bibitem[V]{V}
N. Vonessen,
Actions of Linearly Reductive Groups on Affine PI-Algebras,
Mem. Amer. Math. Soc. {\bf 414}, 1989.

\bibitem[W]{W}
H. Weyl,
The Classical Groups, Their Invariants and Representations,
Princeton Mathematical Series, 1, Princeton University Press, Princeton, NJ, 1946.
Reprint in Princeton Landmarks in Mathematics, 1997.

\bibitem[Z]{Z}
E.I. Zel'manov,
On Engel Lie algebras,
Sibirsk. Mat. Zh. {\bf 29} (1988), No. 5, 112-117 (Russian).
Translation: Siberian Math. J. {\bf 29} (1988), 777-781.

\end{thebibliography}
\end{document}